 \documentclass[10pt,reqno]{amsart}

 \usepackage{mathtools}
 \usepackage{amsthm,amssymb,mathrsfs,enumitem,wasysym}
 
 \usepackage{esint}
 
 \usepackage[numbers,sort&compress]{natbib}
 \usepackage{mathscinet}

 \numberwithin{equation}{section}
 \numberwithin{figure}{section}
 \theoremstyle{plain}
 \newtheorem{thm}{Theorem}[section]
   \theoremstyle{definition}
   \newtheorem{defn}[thm]{Definition}
   \newtheorem*{defn*}{Definition}
   \theoremstyle{remark}
   \newtheorem*{rem*}{Remark}
   \theoremstyle{plain}
   
   \theoremstyle{plain}
   \newtheorem{prop}[thm]{Proposition}
   
   \theoremstyle{definition}
   
   \newtheorem{assump}[thm]{Assumption}

 \makeatletter
 \newcommand{\norm}{\@ifstar{\@normb}{\@normi}}
 \newcommand{\@normb}[2]{\left\Vert{#1}\right\Vert_{#2}}
 \newcommand{\@normi}[2]{\Vert{#1}\Vert_{#2}}
 \makeatother

 \global\long\def\Sob#1#2{W^{#1,#2}} 
 \global\long\def\Sobloc#1#2{W^{#1,#2}_{\mathrm{loc}}} 
 \global\long\def\oSob#1#2{{W}^{#1,#2}_0}
 \global\long\def\Leb#1{L^{#1}}

 \newcommand{\action}[1]{\left<#1 \right>}
 \newcommand{\boldb}{\mathbf{b}}
 \newcommand{\boldF}{\mathbf{F}}
 \newcommand{\boldG}{\mathbf{G}}
 \newcommand{\boldu}{\mathbf{u}}

 \DeclareMathOperator{\Div}{div}
 \newcommand{\relphantom}[1]{\mathrel{\phantom{#1}}}
 
 \newcommand{\myd}[1]{\,d{#1}}

  \title{Elliptic equations in divergence form with   drifts in $L^2$} 
 
\author{Hyunwoo Kwon}
 \address{(28187) Department of Mathematics, Republic of Korea Air Force Academy, Postbox 335-2, 635, Danjae-ro   Sangdang-gu, Cheongju-si  Chung\-cheong\-buk-do, Republic of Korea}
 \email{willkwon@sogang.ac.kr;willkwon@afa.ac.kr}
 
 
 \date{\today}
 
 \allowdisplaybreaks

 \title{Elliptic equations in divergence form with   drifts in $L^2$} 


\begin{document}

\begin{abstract}
We consider the Dirichlet problem for second-order linear elliptic equations in divergence form 
\begin{equation*}
-\Div(A\nabla u)+\boldb \cdot \nabla u+\lambda u=f+\Div \boldF\quad \text{in } \Omega\quad\text{and}\quad  u=0\quad \text{on } \partial\Omega,
\end{equation*}
in bounded Lipschitz domain $\Omega$ in $\mathbb{R}^2$, where $A:\mathbb{R}^2\rightarrow \mathbb{R}^{2^2}$, $\boldb : \Omega\rightarrow \mathbb{R}^2$, and $\lambda \geq 0$ are given. If $2<p<\infty$ and $A$ has a small mean oscillation in small balls, $\Omega$ has small Lipschitz constant, and  $\Div A,\,\boldb \in \Leb{2}(\Omega;\mathbb{R}^2)$, then we prove existence and uniqueness of weak solutions in $\oSob{1}{p}(\Omega)$ of the problem. Similar result also holds for the dual problem. 
\end{abstract}

\maketitle

 \section{Introduction}
This paper is devoted to complementing known results on $\Sob{1}{p}$-estimates for second-order linear elliptic equations with singular drifts terms. Let $\Omega$ be a bounded Lipschitz domain in $\mathbb{R}^n$, $n\geq 2$. For a fixed constant $\lambda \geq 0$ and a given vector field $\boldb=(b^1,b^2,\dots,b^n) : \Omega\rightarrow \mathbb{R}^n$, we consider the following Dirichlet problems of linear  elliptic equations of   second-order: 
\begin{equation}\label{eq:nondivergence-type}\tag{$D$}
\left\{
\begin{alignedat}{2}
-\Div(A\nabla u) +\boldb \cdot \nabla u+\lambda u&=f+\Div \boldF  &&\quad \text{in } \Omega, \\
u&=0 &&\quad \text{on } \partial\Omega.
\end{alignedat}
\right.
\end{equation}
and
\begin{equation}\label{eq:divergence-type}\tag{$D'$}
\left\{
\begin{alignedat}{2}
-\Div(A^T\nabla v)-\Div(v\boldb)+\lambda v&=g+\Div \boldG  &&\quad \text{in } \Omega, \\
v&=0 &&\quad \text{on } \partial\Omega,
\end{alignedat}
\right.
\end{equation}
Here $A=(a^{ij}):\mathbb{R}^n\rightarrow \mathbb{R}^{n\times n}$ denotes an $n\times n$ real matrix-valued measurable function  which is uniformly elliptic, that is, there exists $0<\delta<1$ such that 
\begin{equation}\label{eq:uniformly-elliptic} 
\delta|\xi|^2 \leq \sum_{i,j=1}^n a^{ij}(x)\xi_i \xi_j\quad \text{and}\quad \max_{1\leq i,j\leq n} |a^{ij}(x)|\leq \delta^{-1}\quad    \text{for all } x,\xi \in \mathbb{R}^n. \end{equation}

{$\Sob{1}{p}$-estimates for the problems \eqref{eq:nondivergence-type} and \eqref{eq:divergence-type} were established by several authors under various assumptions on the leading coefficients $a^{ij}$ and the domains $\Omega$ when $\boldb =\mathbf{0}$ or more generally $\boldb \in \Leb{\infty}(\Omega;\mathbb{R}^n)$; see \cite{GT,AQ02,B05,F96,DK11,DK11-3,DK10,K07} and references therein. Also, see the recent survey paper of Dong \cite{D20}.  In particular, Dong-Kim \cite{DK11-3} proved $\Sob{1}{p}$-estimates for the problems \eqref{eq:nondivergence-type} and \eqref{eq:divergence-type} when the leading coefficients satisfy small mean oscillations in small balls and $\boldb\in \Leb{\infty}(\Omega;\mathbb{R}^n)$ on a bounded Lipschitz domain with small Lipschitz constant, see Assumptions \ref{assump:leading-coefficient} and \ref{assump:domain-regularity} for precise statements. One may ask whether we can obtain $\Sob{1}{p}$-estimates for the problems \eqref{eq:nondivergence-type} and \eqref{eq:divergence-type} with unbounded drifts. 

Suppose that $\boldb \in \Leb{q}(\Omega;\mathbb{R}^n)$, where  $n\leq q<\infty$ if $n\geq 3$ and $2<q<\infty$ if $n=2$. In Ladyzhenskaya-Ural'tseva \cite[Chapter 3]{LU68},  they considered the following Dirichlet problem 
\[   -\Div(A\nabla u +u\boldb)+\mathbf{c}\cdot \nabla u+du=f+\Div \boldF\quad \text{in } \Omega,\quad u=0\quad \text{on } \partial\Omega \]
where  $A$ satisfies \eqref{eq:uniformly-elliptic}, $\boldb,\mathbf{c} \in \Leb{q}(\Omega;\mathbb{R}^n)$, $d\in \Leb{q/2}(\Omega)$, and $\Omega$ is a bounded Lipschitz domain in $\mathbb{R}^n$.   It was shown that under some restricted condition on the zeroth order term $d$, for every $f \in \Leb{2\hat{n}/(\hat{n}+2)}(\Omega)$ and $\boldF \in \Leb{2}(\Omega;\mathbb{R}^n)$, there exists a unique weak solution $u\in \oSob{1}{2}(\Omega)$ of the problem. Here $\hat{n}=n$ if $n\geq 3$ and $\hat{n}=2+\varepsilon$ if $n=2$. Stampacchia \cite{S65} also considered a similar problem, but the result is similar to that of Ladyzhenskaya-Ural'tseva. 

To the best knowledge of the author, Trudinger \cite{T73} first proved that given $\lambda \geq 0$, $f\in \Leb{2}(\Omega)$, and $\boldF \in \Leb{2}(\Omega;\mathbb{R}^n)$, there exists a unique weak solution $u$ in $\oSob{1}{2}(\Omega)$ for the problem \eqref{eq:nondivergence-type}. The key tools to prove the theorem are the weak maximum principle and the Fredholm alternative theorem. Later, Droniou \cite{D02} gave another proof by showing $\Sob{1}{2}$-estimates for the problem \eqref{eq:divergence-type} and duality method.  This result was extended by Kim-Kim \cite{KK15} who proved  $\Sob{1}{p}$-estimates for the problem \eqref{eq:nondivergence-type} with $\lambda=0$ when $A$ is the identity matrix, $q'<p<\infty$, and $\Omega$ is a bounded $C^1$-domain. Later, Kang-Kim \cite{KK17} proved $\Sob{1}{p}$-estimates for the problem \eqref{eq:nondivergence-type} when $q'<p<\infty$, $A$ has small mean oscillation, and $\Omega$ is a bounded Lipschitz domain with small Lipschitz constant.  Similar results also hold for the problem \eqref{eq:divergence-type}. We also mention that there are some recent results on the problems  \eqref{eq:nondivergence-type} and \eqref{eq:divergence-type} when the drift $\boldb$ is in weak $L^n$-space; see Moscariello \cite{M11}, Kim-Tsai \cite{KT20}, and the recent result of the author \cite{Kwon21}.

The purpose of this paper is to complement $\Sob{1}{p}$-results on elliptic equations with the drift  $\boldb \in \Leb{2}(\Omega;\mathbb{R}^2)$, which were not considered in \cite{KK15,KK17,KK19}. We remark that our result is   new even if  $A$ is the identity matrix. The motivation for writing this paper is the recent paper due to Krylov \cite{Krylov21}  who proved $\Sob{2}{p}$-result for second-order elliptic equations of non-divergence form with the drift $\boldb$ in $\Leb{n}(\Omega;\mathbb{R}^n)$, $n\geq 2$. More precisely, if $1<p<n$, $\Omega$ is a bounded $C^{1,1}$-domain, $A$ has small mean oscillation (see Assumption \ref{assump:leading-coefficient}), $\boldb \in \Leb{n}(\Omega;\mathbb{R}^n)$, and $\lambda\geq 0$, then there exists a unique $u\in \oSob{1}{p}(\Omega)\cap \Sob{2}{p}(\Omega)$ satisfying 
\[   \sum_{i,j=1}^n a^{ij}D_{ij}u +\boldb \cdot \nabla u -\lambda u =f\quad \text{in } \Omega.\]
Related to our paper, this is the first result on solvability of elliptic equations with $\boldb \in \Leb{2}(\Omega;\mathbb{R}^2)$. Motivated by this result, one may consider $\Sob{1}{p}$-results for the problems \eqref{eq:nondivergence-type} and \eqref{eq:divergence-type} when $\boldb \in \Leb{2}(\Omega;\mathbb{R}^2)$. 

In this paper, it will be shown in Theorem \ref{thm:main-theorem} that if $2<p<\infty$, $A$ has a small mean oscillation in small balls, $\Div A,\boldb \in \Leb{2}(\Omega;\mathbb{R}^2)$, and $\Omega$ has small Lipschitz constant, then for every $\lambda \geq 0$, $f\in \Leb{p}(\Omega)$, and $\boldF\in \Leb{p}(\Omega;\mathbb{R}^2)$, there exists a unique weak solution $u\in \oSob{1}{p}(\Omega)$ of the problem \eqref{eq:nondivergence-type}.   By duality, we have a similar result for the problem \eqref{eq:divergence-type}, see Section \ref{sec:main} for the precise statements and the definition of $\Div A \in \Leb{2}(\Omega;\mathbb{R}^2)$.  

Our method to prove Theorem \ref{thm:main-theorem} is to use a functional analytic argument as in \cite{KK17,KK18,KK19}. A key tool is the following $\varepsilon$-inequality (Proposition \ref{prop:Gerhardt}) inspired by Gerhardt \cite{G79}: suppose that $\Omega$ is a bounded  Lipschitz domain in $\mathbb{R}^2$ and $2<p<\infty$.  Then for each $\varepsilon>0$,  there exists a constant $C_\varepsilon=C_\varepsilon(\varepsilon,p,\boldb,\Omega)>0$ such that 
\begin{equation}\label{eq:eps-inequality}
  \norm{\boldb\cdot \nabla u}{\Sob{-1}{p}(\Omega)}\leq \varepsilon \norm{\nabla u}{\Leb{p}(\Omega)} +C_\varepsilon \norm{u}{\Leb{p}(\Omega)} 
  \end{equation}
for all $u\in \Sob{1}{p}(\Omega)$. Using this estimate, we prove that if $2<p<\infty$, $A$ has small mean oscillations on small balls and $\Omega$ has small Lipschitz constant, then for sufficiently large $\lambda_1$, the following holds for  $\lambda \geq \lambda_1$: if $f\in \Leb{p}(\Omega)$, and $\boldF\in \Leb{p}(\Omega;\mathbb{R}^2)$, then there exists a unique weak solution $u\in \oSob{1}{p}(\Omega)$ of the problem  \eqref{eq:nondivergence-type}. Similar results also hold for the problem \eqref{eq:divergence-type}. This result induces an operator $\mathcal{L}_p+\lambda_1 \mathcal{I}_{p}:\Sob{-1}{p}(\Omega)\rightarrow \oSob{1}{p}(\Omega)$ whose inverse can be regarded as a compact operator on $\Leb{p}(\Omega)$. Hence by the Fredholm alternative theorem (see 	\cite[Theorem 6.6]{Brezis11} e.g.), it suffices to prove the uniqueness of weak solutions in $\oSob{1}{p}(\Omega)$ for the problem \eqref{eq:nondivergence-type}, see Section \ref{sec:completion} for the definition of $\mathcal{L}_p+\lambda_1 \mathcal{I}_{p}$ and the reduction. To show the uniqueness of weak solutions in $\oSob{1}{p}(\Omega)$ of the problem \eqref{eq:nondivergence-type}, we use an Alexsandrov type maximum principle, which was recently proved by Krylov \cite[Corollary 3.1]{krylov2020stochastic}, see Theorem \ref{thm:Alexsandrov}. To use this theorem in our setting, we assume in addition that $\Div A \in \Leb{2}(\Omega;\mathbb{R}^2)$ to convert an elliptic equation in divergence form into an equation in non-divergence from. It seems to be open whether we can remove the additional assumption  $\Div A\in \Leb{2}(\Omega;\mathbb{R}^2)$. 

The organization of this paper is as follows. We introduce some notations and state the main theorem in the next section. In Section \ref{sec:large-lambda}, we prove $\Sob{1}{p}$-results for the problems \eqref{eq:nondivergence-type} and \eqref{eq:divergence-type} for sufficiently large $\lambda$. Next, we prove the uniqueness of  weak solutions of the problem \eqref{eq:nondivergence-type} in Section \ref{sec:uniqueness}. Proof of the main theorem is presented in Section \ref{sec:completion}.   For the reader's convenience, we give all necessary details that can be found in \cite{KK17,KK18,KK19}.  

 \subsection*{Acknowledgements}
The author would like to thank Prof.\ Doyoon Kim for introducing the recent result of Krylov \cite{krylov2020stochastic} to the author and for a valuable discussion. Also, the author would like to thank the advisor Prof. Hyunseok Kim for comments on previous drafts. Finally, the author would like to  thank the anonymous referee for the careful reading of the manuscript and for giving useful comments and suggestions
to improve the paper.

\section{Notation and Main result}\label{sec:main}
In this section, we introduce several notations used in this article. Also, we give the main theorem of this paper. We use ``$:=$'' to denote a definition. As usual, $\mathbb{R}^n$ stands the standard Euclidean space of $n$-points and $|\cdot|$ is the standard Euclidean norm on $\mathbb{R}^n$. For $r>0$ and $x\in \mathbb{R}^n$, we write $B_r(x):=\{ y \in \mathbb{R}^n : |x-y|<r\}$. We also write $B_r:=B_r(0)$. For $x\in \mathbb{R}^n$, we write $x=(x',x_n)$ where $x'\in \mathbb{R}^{n-1}$ and $B_r'(x'):=\{ y' \in \mathbb{R}^{n-1} : |x'-y'|<r\}$. For $1\leq j,k\leq n$, we denote 
\[   D_j u = \frac{\partial u}{\partial x_j},\quad D_{kj} u= D_j D_k u=u_{x_k x_j}.\]
We also use the notation $\nabla u:=(D_1u ,\dots, D_n u)$ for the gradient of $u$.  

We denote by $X'$ the dual space of a Banach space $X$. The dual pairing of $X$ and $X'$ is denoted by $\action{\cdot,\cdot}_{X',X}$ or simply $\action{\cdot,\cdot}$.  For $k\in \mathbb{N}\cup\{\infty\}$, let $C_c^k(\Omega)$ be the space of all functions in $C^k(\mathbb{R}^n)$ with compact supports in $\Omega$ and let $C^k(\overline{\Omega})$ the space of the restrictions to $\overline{\Omega}$ of all functions in $C^k(\mathbb{R}^n)$. For $k\in \mathbb{N}$ and $1\leq p<\infty$, $\Leb{p}(\Omega)$ and $\Sob{k}{p}(\Omega)$ denote the standard $\Leb{p}$-space on $\Omega$ with Lebesgue measure and the Sobolev spaces on $\Omega$, respectively. We define $\oSob{1}{p}(\Omega)$ the closure of $C_c^\infty(\Omega)$ in $\Sob{1}{p}(\Omega)$. For $1<p<\infty$, we write $p':=p/(p-1)$ the conjugate exponent to $p$. For such $p$, we define $\Sob{-1}{p}(\Omega):=(\oSob{1}{p'}(\Omega))'$. For $1\leq p<n$, $p^*:=\frac{np}{n-p}$  the Sobolev exponent.  For $1\leq p<\infty$, by $\Leb{p}(\Omega;\mathbb{R}^m)$, we denote the set of all $\mathbb{R}^m$-valued measurable functions $\mathbf{u}=(u^1,\dots,u^m)$ on $\Omega$ satisfying 
\[   \norm{\boldu}{\Leb{p}(\Omega)} :=\left(\int_\Omega |\boldu(x)|^p \myd{x}\right)^{1/p}<\infty.\]  
Similarly, $C_c^\infty(\Omega;\mathbb{R}^m)$ denotes the set of all $\mathbb{R}^m$-valued measurable functions $\Phi=(\phi^1,\dots,\phi^m)$ on $\Omega$ satisfying $\phi^i \in C_c^\infty(\Omega)$ for all $1\leq i\leq m$. 

For open sets $U$ and $V$, we write $V\Subset U$ if $\overline{V}$ is compact and $\overline{V}\subset U$. For $1\leq p<\infty$ and $k\in \mathbb{N}\cup\{0\}$, we write $\Sobloc{k}{p}(\Omega)$ if $u:\Omega \rightarrow \mathbb{R}$ satisfy $u \in \Sob{k}{p}(\Omega')$ for any  $\Omega'\Subset \Omega$. Similarly, a vector field $\boldu :\Omega\rightarrow \mathbb{R}^m$ is in $\Leb{p}_{\mathrm{loc}}(\Omega;\mathbb{R}^m)$ if $\boldu \in \Leb{p}(\Omega';\mathbb{R}^m)$ for any $\Omega'\Subset \Omega$. For a measurable function $f$ on $E\subset \mathbb{R}^n$, we write 
\[   (f)_{E} = \frac{1}{|E|} \int_{E} f \myd{x}=\fint_{E} f\myd{x},\]
where $|E|$ denotes the $n$-dimensional Lebesgue measure of $E$.  Finally, by  $C=C(p_1,\dots,p_k)$, we denote a generic positive constant depending only on the parameters $p_1,\dots,p_k$.

We define weak solutions of the problem \eqref{eq:nondivergence-type} and \eqref{eq:divergence-type} as follows.
\begin{defn}
Let $\lambda \geq0$, $1<p<\infty$, and $\boldb : \Omega \rightarrow \mathbb{R}^n$ be a given measurable vector field. 
\begin{enumerate}
\item Given $f\in \Leb{p}(\Omega)$ and $\boldF\in \Leb{p}(\Omega;\mathbb{R}^n)$, we say that $u\in \oSob{1}{p}(\Omega)$ is a \emph{weak solution} of \eqref{eq:nondivergence-type} if $\boldb \cdot \nabla u \in \Leb{1}_{\mathrm{loc}}(\Omega)$ and 
\begin{equation}\label{eq:weak-sol-1}
  \int_\Omega A\nabla u \cdot \nabla \phi \myd{x}+\int_\Omega (\boldb \cdot \nabla u)\phi \myd{x}+\lambda \int_\Omega u\phi \myd{x}=\int_\Omega f \phi \myd{x}-\int_\Omega \boldF \cdot \nabla \phi \myd{x} 
\end{equation} 
for all $\phi \in C_c^\infty(\Omega)$.  
\item Given $g\in \Leb{p'}(\Omega)$ and $\boldG\in \Leb{p'}(\Omega;\mathbb{R}^n)$, we say that $v\in \oSob{1}{p'}(\Omega)$ is a \emph{weak solution} of \eqref{eq:divergence-type} if $v\boldb   \in \Leb{1}_{\mathrm{loc}}(\Omega;\mathbb{R}^n)$ and 
\begin{equation}\label{eq:weak-sol-2}
\int_\Omega (A^T\nabla v +v\boldb)\cdot \nabla \psi \myd{x}+\lambda \int_\Omega v\psi \myd{x}=\int_\Omega g \psi \myd{x}-\int_\Omega \boldG \cdot \nabla \psi \myd{x} 
\end{equation} 
for all $\psi \in C_c^\infty(\Omega)$.     
\end{enumerate}
\end{defn}  
Let $\boldb \in \Leb{n}(\Omega;\mathbb{R}^n)$. If $n'\leq p<\infty$, then by H\"older's inequality,  we have $\boldb \cdot \nabla v \in \Leb{1}(\Omega)$ for any $v \in \Sob{1}{p}(\Omega)$. If $1<p\leq n$, then it follows from H\"older's inequality and Sobolev's embedding theorem that 
\[   \norm{v\boldb}{\Leb{1}(\Omega)}\leq \norm{\boldb}{\Leb{n}(\Omega)}\norm{v}{\Leb{n'}(\Omega)}\leq C\norm{\boldb}{\Leb{n}(\Omega)}\norm{v}{\Sob{1}{p}(\Omega)} \]
for all $v\in \Sob{1}{p}(\Omega)$. Hence if we have an unbounded drift  $\boldb \in \Leb{n}(\Omega;\mathbb{R}^n)$, then the range of $p$ is limited to ensure the well-definedness of weak solutions in $\oSob{1}{p}(\Omega)$ for  problems \eqref{eq:nondivergence-type} and \eqref{eq:divergence-type}, respectively.

We impose the following regularity assumption on the leading coefficients: 
\begin{assump}[$\gamma$]\label{assump:leading-coefficient}
There exists a constant $R_0 \in (0,1]$ such that 
\[\max_{1\leq i,j\leq n} \fint_{B_r(x)} \left|a^{ij}(y)-(a^{ij})_{B_r(x)} \right|\myd{y}\leq \gamma \]
 for any $x\in \mathbb{R}^n$ and $0<r\leq R_0$.
\end{assump} 
Note that if  $A$ is in VMO (see e.g. \cite{K07}), then Assumption \ref{assump:leading-coefficient} $(\gamma)$ is satisfied for any $\gamma>0$. 

 Next, we impose the following regularity assumption on the boundary of the domain $\Omega$: 
\begin{assump}[$\theta$]\label{assump:domain-regularity}
There is a constant $R_0\in (0,1]$ such that for any $x_0 \in \partial\Omega$, there exists a Lipschitz function $\sigma : \mathbb{R}^{n-1}\rightarrow \mathbb{R}$ such that 
\[  \Omega \cap B_{R_0}(x_0) =\{ x \in B_r(x_0): x_n>\sigma(x') \} \]
and 
\[   \sup_{x',y'\in B_{R_0}'(x_0'), x'\neq y'} \frac{|\sigma(x')-\sigma(y')|}{|x'-y'|} \leq \theta \]
in some coordinate system. 
\end{assump}
It is easy to check that every $C^1$-domain satisfies Assumption \ref{assump:domain-regularity} $(\theta)$ for any $\theta>0$. 

To state our main theorem, we introduce the definition of weak $L^2$-divergence for a bounded measurable matrix-valued function. 
\begin{defn}
A bounded measurable matrix-valued function $A:\mathbb{R}^2\rightarrow \mathbb{R}^{2\times 2}$ has \emph{weak $\Leb{2}$-divergence in $\Omega$} if  there exists a vector field $\mathbf{c}$ in $\Leb{2}(\Omega;\mathbb{R}^2)$ such that 
\[   \int_\Omega \sum_{i,j=1}^2 a^{ij} D_i \phi^j dx=-\int_\Omega \mathbf{c}  \cdot \Phi \myd{x} \]
for all  $\Phi=(\phi^1,\phi^2)\in C_c^\infty(\Omega;\mathbb{R}^2)$. In this case, we write $\mathbf{c}=\Div A$ and $\Div A \in \Leb{2}(\Omega;\mathbb{R}^2)$. 
\end{defn}

Now we state the main theorem of this paper.  
\begin{thm}\label{thm:main-theorem}
Let $\Omega$ be a bounded domain in $\mathbb{R}^2$,  $2<p<\infty$, and $\lambda \geq 0$. Suppose that  $A$ satisfies \eqref{eq:uniformly-elliptic}, $\Div A \in \Leb{2}(\Omega;\mathbb{R}^2)$, and $\boldb \in \Leb{2}(\Omega;\mathbb{R}^2)$.  Then there exist constants $\gamma=\gamma(\delta, p, \norm{\Div A}{\Leb{2}(\Omega)},\norm{\boldb}{\Leb{2}(\Omega)})$  and $\theta=\theta(\delta,p)>0$   such that under Assumptions \ref{assump:leading-coefficient} $(\gamma)$ and  \ref{assump:domain-regularity} $(\theta)$, the following hold: 
\begin{enumerate}
\item[\rm (i)] For every $f\in \Leb{p}(\Omega)$ and $\boldF\in \Leb{p}(\Omega;\mathbb{R}^2)$, there exists a unique weak solution $u\in \oSob{1}{p}(\Omega)$ of \eqref{eq:nondivergence-type}. Moreover we have 
\[      \norm{\nabla u}{\Leb{p}(\Omega)}+\lambda^{1/2}\norm{u}{\Leb{p}(\Omega)}\leq C \left[\min \left(1,\lambda^{-1} \right)\norm{f}{\Leb{p}(\Omega)}+\norm{\boldF}{\Leb{p}(\Omega)} \right]  \]
for some constant $C$ independent of $u$, $f$, $\boldF$, and $\lambda$.
\item[\rm (ii)] For every $g\in \Leb{p'}(\Omega)$ and $\boldG\in \Leb{p'}(\Omega;\mathbb{R}^2)$, there exists a unique weak solution $v\in \oSob{1}{p'}(\Omega)$ of \eqref{eq:divergence-type}. Moreover we have 
\[      \norm{\nabla v}{\Leb{p'}(\Omega)}+\lambda^{1/2}\norm{v}{\Leb{p'}(\Omega)}\leq C \left[\min \left(1,\lambda^{-1} \right)\norm{g}{\Leb{p'}(\Omega)}+\norm{\boldG}{\Leb{p'}(\Omega)} \right] \]
for some constant $C$ independent of $v$, $g$, $\boldG$, and $\lambda$.
\end{enumerate}
\end{thm}
\begin{rem*}
(i) Theorem \ref{thm:main-theorem} complements   the result of Kim-Kim \cite{KK15} when $a^{ij}=\delta^{ij}$ and    $\boldb \in \Leb{2}(\Omega;\mathbb{R}^2)$, where $\delta^{ij}$ is the Kronecker delta.  

(ii) The range $2<p<\infty$ in Theorem \ref{thm:main-theorem} is optimal. Define 
\[ a^{ij}=\delta^{ij},\quad  u(x)= \ln |\ln |x|| ,\quad\text{and}\quad \boldb(x)=-\frac{x}{|x|^2\ln |x|}. \]
Then $\boldb \in \Leb{2}(B_{1/e};\mathbb{R}^2)$ and $u\in \oSob{1}{2}(B_{1/e})$ is a nontrivial weak solution satisfying the problem \eqref{eq:nondivergence-type} with $\lambda =0$, $f=0$, and $\boldF=\mathbf{0}$. Hence the uniqueness of weak solutions in $\oSob{1}{2}(\Omega)$ of the problem  \eqref{eq:nondivergence-type} fails in general. More related examples can be found in Filonov \cite{F13} and Filonov-Shilkin \cite{FS18}. 

(iii) In Kang-Kim \cite[Theorem 2.5]{KK17}, they obtained $\Sob{1}{p}$-estimates for the problems \eqref{eq:nondivergence-type} and \eqref{eq:divergence-type} when $a^{ij}$ has small mean oscillations in small balls and $\boldb \in \Leb{q}(\Omega;\mathbb{R}^2)$ with $q>2$.  It seems to be open whether we can remove additional assumption $\Div A \in \Leb{2}(\Omega;\mathbb{R}^2)$ when $\boldb \in \Leb{2}(\Omega;\mathbb{R}^2)$.  
\end{rem*}

\section{Solvability of the problems \eqref{eq:nondivergence-type} and \eqref{eq:divergence-type} for large $\lambda$}\label{sec:large-lambda}
In this section, we obtain  $\Sob{1}{p}$-estimates for the problems \eqref{eq:nondivergence-type} and \eqref{eq:divergence-type} for sufficiently large $\lambda$. 

We first show basic estimates for the drift terms and the $\varepsilon$-inequalities inspired by Gerhardt \cite{G79}, which play crucial roles in the proof of the main theorem of this paper. 
\begin{prop}\label{prop:Gerhardt}
Let $\Omega$ be a bounded Lipschitz  domain in $\mathbb{R}^n$, $n\geq 2$, and $n'<p<\infty$. Suppose that $\boldb \in \Leb{n}(\Omega;\mathbb{R}^n)$. Then there exists a constant $C=C(n,p,\Omega)>0$ such that 
\begin{equation}\label{eq:basic-estimate}
   \int_\Omega |(v\boldb) \cdot \nabla u|\myd{x}\leq C\norm{\boldb}{\Leb{n}(\Omega)}\norm{u}{\Sob{1}{p}(\Omega)}\norm{v}{\Sob{1}{p'}(\Omega)}
\end{equation}
for all $u\in \Sob{1}{p}(\Omega)$ and $v\in \Sob{1}{p'}(\Omega)$. 

For each $\varepsilon>0$, there exists a constant $C_\varepsilon=C_\varepsilon(\varepsilon,n,p,\boldb,\Omega)>0$ such that 
\begin{equation}\label{eq:bdotv-e-inequality} 
\boldb \cdot \nabla u \in \Sob{-1}{p}(\Omega)\quad \text{and}\quad  \norm{\boldb\cdot \nabla u}{\Sob{-1}{p}(\Omega)}\leq \varepsilon \norm{u}{\Sob{1}{p}(\Omega)}+C_{\varepsilon}\norm{u}{\Leb{p}(\Omega)} 
\end{equation}
for all $u\in \Sob{1}{p}(\Omega)$.  Similarly, for each $\varepsilon>0$, there exists a constant $C_\varepsilon^*=C_\varepsilon^*(\varepsilon,n,p,\boldb,\Omega)>0$ such that 
\begin{equation}\label{eq:div-e-inequality} 
\Div(v\boldb)\in \Sob{-1}{p'}(\Omega)\quad \text{and}\quad  \norm{\Div(v\boldb)}{\Sob{-1}{p'}(\Omega)}\leq \varepsilon \norm{v}{\Sob{1}{p'}(\Omega)}+C_{\varepsilon}^*\norm{v}{\Leb{p'}(\Omega)}
\end{equation}
for all $v\in \Sob{1}{p'}(\Omega)$.
\end{prop}
\begin{proof}
By H\"older's inequality and Sobolev's embedding theorem, we have 
\begin{align*} 
  \int_\Omega |(v\boldb)\cdot \nabla u|\myd{x}&\leq \norm{\boldb}{\Leb{n}(\Omega)}\norm{\nabla u}{\Leb{p}(\Omega)}\norm{v}{\Leb{(p')^*}(\Omega)} \\
  &\leq C(n,p,\Omega) \norm{\boldb}{\Leb{n}(\Omega)}\norm{u}{\Sob{1}{p}(\Omega)}\norm{v}{\Sob{1}{p'}(\Omega)}.
  \end{align*}
  for all $u\in \Sob{1}{p}(\Omega)$ and $v\in \Sob{1}{p'}(\Omega)$. This proves \eqref{eq:basic-estimate}. 
  
Note that $\boldb \cdot \nabla u\in \Sob{-1}{p}(\Omega)$ by the estimate \eqref{eq:basic-estimate}. To show $\varepsilon$-inequality \eqref{eq:bdotv-e-inequality}, note that integration by part shows that the identity
 \begin{equation}\label{eq:integration-by-part}
 \int_\Omega (\boldb \cdot \nabla u) v dx=-\int_\Omega (\boldb \cdot \nabla v)u \myd{x}-\int_\Omega (\Div \boldb) uv \myd{x}
 \end{equation} 
 holds for any $\boldb \in C^\infty_c(\Omega;\mathbb{R}^n)$, $u\in C^\infty(\overline{\Omega})$, and $v\in C_c^\infty(\Omega)$. Since $C^\infty(\overline{\Omega})$ is dense in $\Sob{1}{p}(\Omega)$ (see \cite[Theorem 4.3]{EG}) and the estimate \eqref{eq:basic-estimate} holds, a standard density argument shows that identity \eqref{eq:integration-by-part} holds for any $\boldb \in C_c^\infty({\Omega};\mathbb{R}^n)$, $u\in \Sob{1}{p}(\Omega)$, and $v\in C_c^\infty(\Omega)$. 
  
Let $\varepsilon>0$ be given. Since $C^\infty_c(\Omega;\mathbb{R}^n)$ is dense in $\Leb{n}(\Omega;\mathbb{R}^n)$, there exists $\boldb_{\varepsilon} \in C^\infty_c(\Omega;\mathbb{R}^n)$ such that 
\[   \norm{\boldb_{\varepsilon}-\boldb}{\Leb{n}(\Omega)}<{\varepsilon}/{C},\]
where $C$ is the same constant in \eqref{eq:basic-estimate}. Fix $v\in C_c^\infty(\Omega)$. Then by \eqref{eq:integration-by-part}, we have
\begin{align*}
\int_\Omega (\boldb\cdot \nabla u)v\myd{x}&=\int_\Omega [(\boldb-\boldb_{\varepsilon})\cdot \nabla u] v \myd{x}+\int_\Omega (\boldb_\varepsilon\cdot \nabla u)v dx \\
&=\int_\Omega [(\boldb-\boldb_\varepsilon)\cdot \nabla u] v \myd{x}-\int_\Omega (\boldb_\varepsilon \cdot \nabla v) u dx-\int_\Omega (\Div \boldb_\varepsilon)uv \myd{x}.
\end{align*}
By \eqref{eq:basic-estimate} and H\"older's inequality, we get 
\begin{align*}
\int_\Omega |(\boldb \cdot \nabla u)v |\myd{x}&\leq C\norm{\boldb-\boldb_{\varepsilon}}{\Leb{n}(\Omega)}\norm{u}{\Sob{1}{p}(\Omega)}\norm{v}{\Sob{1}{p'}(\Omega)}\\
&\relphantom{=}  +\norm{\boldb_\varepsilon}{\Leb{\infty}(\Omega)}\norm{u}{\Leb{p}(\Omega)}\norm{v}{\Leb{p'}(\Omega)}+\norm{\Div \boldb_\varepsilon}{\Leb{\infty}(\Omega)}\norm{u}{\Leb{p}(\Omega)}\norm{v}{\Leb{p'}(\Omega)}\\
&\leq \left(\varepsilon \norm{u}{\Sob{1}{p}(\Omega)}+C_\varepsilon \norm{u}{\Leb{p}(\Omega)}\right)\norm{v}{\Sob{1}{p'}(\Omega)},
\end{align*}
where $C_\varepsilon = \norm{\boldb_{\varepsilon}}{\Leb{\infty}(\Omega)}+\norm{\Div \boldb_{\varepsilon}}{\Leb{\infty}(\Omega)}$. Since $v\in C_c^\infty(\Omega)$ was arbitrary chosen, this implies that 
\[  \norm{\boldb\cdot \nabla u}{\Sob{-1}{p}(\Omega)}\leq \varepsilon \norm{u}{\Sob{1}{p}(\Omega)}+C_\varepsilon \norm{u}{\Leb{p}(\Omega)},\]
which proves \eqref{eq:bdotv-e-inequality}. 

To show \eqref{eq:div-e-inequality}, suppose that $u\in C^\infty_c(\Omega)$ and $v\in \Sob{1}{p'}(\Omega)$. Since
\[  \int_\Omega (v\boldb)\cdot \nabla u \myd{x}=\int_\Omega [v(\boldb-\boldb_\varepsilon)]\cdot \nabla u \myd{x}+\int_\Omega (v\boldb_{\varepsilon }) \cdot \nabla u \myd{x},
 \]
it follows from \eqref{eq:basic-estimate} and H\"older's inequality that 
 \begin{equation}\label{eq:e-inequality}
    \int_\Omega |(v\boldb)\cdot \nabla u| \myd{x}\leq (\varepsilon \norm{v}{\Sob{1}{p'}(\Omega)}+\norm{\boldb_\varepsilon}{\Leb{\infty}(\Omega)}\norm{v}{\Leb{p'}(\Omega)})\norm{u}{\Sob{1}{p}(\Omega)}.  
\end{equation}
 Since 
 \[  \action{\Div(v\boldb),u}=-\int_\Omega (v\boldb)\cdot \nabla u \myd{x} \]
 for all $u\in C_c^\infty(\Omega)$, it follows from \eqref{eq:e-inequality} that $\Div (v\boldb)\in \Sob{-1}{p'}(\Omega)$ and the estimate  \eqref{eq:div-e-inequality} holds. This completes the proof of Proposition \ref{prop:Gerhardt}. 
\end{proof}

We use the following special case of Dong-Kim \cite[Theorem 7]{DK11-3}. 
\begin{thm}\label{thm:no-first-order}
Let $1<p<\infty$ and $\Omega$ be a bounded domain in $\mathbb{R}^n$, $n\geq 2$. Then there exist $\gamma=\gamma(n,p,\delta)>0$, $\theta=\theta(n,p,\delta)>0$, and $\lambda_0=\lambda_0(n,p,\delta,R_0)\geq 1$ such that under Assumptions \ref{assump:leading-coefficient} $(\gamma)$ and \ref{assump:domain-regularity} $(\theta)$, the following  holds for any $\lambda \geq \lambda_0$:  for any $f\in \Leb{p}(\Omega;\mathbb{R}^n)$ and $\boldF \in \Leb{p}(\Omega;\mathbb{R}^n)$, there exists a unique weak solution $u\in \oSob{1}{p}(\Omega)$ such that 
\[    -\Div(A \nabla u) +\lambda u=f+\Div \boldF\quad \text{in } \Omega. \]
Moreover  we have 
\[   \lambda^{1/2}\norm{u}{\Leb{p}(\Omega)}+\norm{\nabla u}{\Leb{p}(\Omega)}\leq C\left(\lambda^{-1/2}\norm{f}{\Leb{p}(\Omega)}+\norm{\boldF}{\Leb{p}(\Omega)} \right)  \]
for some constant $C=C(n,p,\delta,R_0,\Omega)>0$. 
\end{thm}
\begin{rem*}
Let $\Omega$ be a bounded  domain in $\mathbb{R}^n$, $n\geq 2$, and $1<p<\infty$. If $f_0 \in \Sob{-1}{p}(\Omega)$, then there exists $\boldF_0 \in \Leb{p}(\Omega;\mathbb{R}^n)$ such that 
\[   \Div\boldF_0=f_0\quad \text{in } \Omega \quad \text{and}\quad   \norm{\boldF_0}{\Leb{p}(\Omega)}\leq C(n,p,\Omega)\norm{f_0}{\Sob{-1}{p}(\Omega)},  \]
(see e.g. \cite[Lemma 3.9]{KT20}). By Theorem \ref{thm:no-first-order}, there exist  $\gamma=\gamma(n,p,\delta)>0$, $\theta=\theta(n,p,\delta,R_0)>0$, and $\lambda_0=\lambda_0(n,p,\delta,R_0)\geq 1$ such that under Assumptions \ref{assump:leading-coefficient} $(\gamma)$ and \ref{assump:domain-regularity} $(\theta)$, we have for $\lambda \geq \lambda_0$ and for each $f\in \Leb{p}(\Omega)$, there exists a unique weak solution $u\in \oSob{1}{p}(\Omega)$ such that 
\[   \int_\Omega (A \nabla u) \cdot \nabla \phi \myd{x}+\lambda \int_\Omega u \phi \myd{x}=\int_\Omega f \phi \myd{x}-\int_\Omega \boldF_0\cdot \nabla \phi \myd{x}\]
for all $\phi \in C_c^\infty(\Omega)$. Moreover  we have 
\begin{equation}\label{eq:a-priori-negative-space}
\begin{aligned}
\lambda^{1/2}\norm{u}{\Leb{p}(\Omega)}+\norm{\nabla u}{\Leb{p}(\Omega)}&\leq C(\lambda^{-1/2}\norm{f}{\Leb{p}(\Omega)}+\norm{\boldF_0}{\Leb{p}(\Omega)})\\
   &\leq C(\lambda^{-1/2}\norm{f}{\Leb{p}(\Omega)}+\norm{f_0}{\Sob{-1}{p}(\Omega)} )
\end{aligned}
\end{equation}
for some constant $C=C(n,p,\delta,R_0,\Omega)>0$. 
\end{rem*}

Now we present the main theorem of this section. From now on, we mainly focus on the case $\boldb \in \Leb{2}(\Omega;\mathbb{R}^2)$ since other cases $\boldb \in \Leb{q}(\Omega;\mathbb{R}^n)$ are already considered in Kang-Kim \cite{KK17} when $n\leq q<\infty$ if $n\geq 3$ and $2<q<\infty$ if $n=2$. 
 \begin{thm}\label{thm:large-lambda}
Let  $2<p<\infty$ and $\Omega$ be a bounded domain in $\mathbb{R}^2$. Suppose that $A$ satisfies \eqref{eq:uniformly-elliptic} and $\boldb \in \Leb{2}(\Omega;\mathbb{R}^2)$. Then there exist $\gamma=\gamma(p,\delta)>0$, $\theta=\theta(p,\delta)>0$, and $\lambda_1=\lambda_1(p,\delta,R_0,\Omega,\boldb)\geq 1$ such that under Assumptions \ref{assump:leading-coefficient} $(\gamma)$ and \ref{assump:domain-regularity} $(\theta)$, the following  results hold for any $\lambda \geq \lambda_1$: 
 \begin{enumerate}
     \item[\rm (i)] If $f\in \Leb{p}(\Omega)$ and $\boldF \in \Leb{p}(\Omega;\mathbb{R}^2)$, then there exists a unique weak solution  $u\in \oSob{1}{p}(\Omega)$ of \eqref{eq:nondivergence-type}. Moreover  
\[   \norm{\nabla u}{\Leb{p}(\Omega)}+\lambda^{1/2}\norm{u}{\Leb{p}(\Omega)}\leq C  \left[\lambda^{-1/2}\norm{f}{\Leb{p}(\Omega)}+\norm{\boldF}{\Leb{p}(\Omega)} \right]\]
for some constant $C=C(p,\delta,R_0,\Omega)>0$. 
     \item[\rm (ii)] If $g\in \Leb{p'}(\Omega)$ and $\boldG \in \Leb{p'}(\Omega;\mathbb{R}^2)$, then there exists a unique weak solution   $v\in \oSob{1}{p'}(\Omega)$ of \eqref{eq:divergence-type}. Moreover  
\[   \norm{\nabla v}{\Leb{p'}(\Omega)}+\lambda^{1/2}\norm{v}{\Leb{p'}(\Omega)}\leq C  \left[\lambda^{-1/2}\norm{g}{\Leb{p'}(\Omega)}+\norm{\boldG}{\Leb{p'}(\Omega)} \right]\]
for some constant $C=C(p,\delta,R_0,\Omega)>0$. 
 \end{enumerate}
 \end{thm}
 \begin{proof}
 Let $f\in \Leb{p}(\Omega)$ and $\boldF \in \Leb{p}(\Omega;\mathbb{R}^2)$. Then $\Div \boldF \in \Sob{-1}{p}(\Omega)$. For each $u\in \oSob{1}{p}(\Omega)$, it follows from Proposition \ref{prop:Gerhardt} that $\boldb \cdot \nabla u \in \Sob{-1}{p}(\Omega)$. Moreover, for each $\varepsilon>0$, there exists a constant $C_\varepsilon =C(\varepsilon,p,\boldb,\Omega)>0$ such that 
 \begin{equation}\label{eq:g-inequality}
\norm{\boldb \cdot \nabla u}{\Sob{-1}{p}(\Omega)}\leq \varepsilon \norm{\nabla u}{\Leb{p}(\Omega)}+C_\varepsilon \norm{u}{\Leb{p}(\Omega)}
\end{equation}
for all $u\in \oSob{1}{p}(\Omega)$.   By the remark of Theorem \ref{thm:no-first-order}, there exist $\gamma=\gamma(p,\delta)>0$, $\theta=\theta(p,\delta)>0$, and $\lambda_0=\lambda_0(p,\delta,R_0,\Omega)\geq 1$ such that under Assumptions \ref{assump:leading-coefficient} $(\gamma)$ and \ref{assump:domain-regularity} $(\theta)$, for $\lambda \geq \lambda_0$, there exists a unique $\overline{u}=\mathcal{T}(u)\in \oSob{1}{p}(\Omega)$ satisfying 
\[  \int_\Omega A\nabla \overline{u}\cdot \nabla \phi \myd{x}+\lambda \int_\Omega \overline{u}\phi\myd{x}=\action{\Div \boldF-\boldb \cdot \nabla u,\phi}+\int_\Omega f \phi \myd{x}\]
for all $\phi \in C_c^\infty(\Omega)$. Moreover, we have 
\[   \lambda^{1/2}\norm{\overline{u}}{\Leb{p}(\Omega)}+\norm{\nabla \overline{u}}{\Leb{p}(\Omega)}\leq C_0\left(\lambda^{-1/2}\norm{f}{\Leb{p}(\Omega)}+\norm{\Div \boldF -\boldb \cdot \nabla u}{\Sob{-1}{p}(\Omega)} \right) \]
for some constant $C_0=C_0(p,\delta,R_0,\Omega)>0$.  Choose $\varepsilon>0$ so that $\varepsilon C_0 =1/2$. Then we have 
\begin{align*}
\lambda^{1/2}\norm{\overline{u}}{\Leb{p}(\Omega)}+\norm{\nabla \overline{u}}{\Leb{p}(\Omega)}
&\leq C_0\lambda^{-1/2}\norm{f}{\Leb{p}(\Omega)}+ C_0 \norm{\boldF}{\Leb{p}(\Omega)}\\
&\relphantom{=}+\frac{1}{2}\left(\norm{\nabla u}{\Leb{p}(\Omega)}+C_* \norm{u}{\Leb{p}(\Omega)}\right),
\end{align*}  
where $C_*=C_*(p,\boldb,\delta,R_0,\Omega)>0$. Moreover, we have 
\begin{align*}   
&\lambda^{1/2}\norm{\mathcal{T}(u_1)-\mathcal{T}(u_2)}{\Leb{p}(\Omega)}+\norm{\nabla (\mathcal{T}(u_1)-\mathcal{T}(u_2))}{\Leb{p}(\Omega)}\\
&\leq \frac{1}{2} \left( \norm{\nabla (u_1-u_2)}{\Leb{p}(\Omega)}+C_*\norm{u_1-u_2}{\Leb{p}(\Omega)}\right)
\end{align*}
for $u_1, u_2 \in \oSob{1}{p}(\Omega)$. Let $\lambda_1:=\max (C_*^2,\lambda_0)+1$. Then for $\lambda \geq \lambda_1$, $\mathcal{T}$ is a contraction on $\oSob{1}{p}(\Omega)$ which is a Banach space equipped with the equivalent norm $\norm{\nabla \cdot}{\Leb{p}(\Omega)}+\norm{\cdot }{\Leb{p}(\Omega)}$. Hence by the Banach fixed point theorem, there exists a unique $u\in \oSob{1}{p}(\Omega)$ such that $\mathcal{T}(u)=u$, that is, $u$ is a weak solution of the problem \eqref{eq:nondivergence-type}. Moreover, $u$ satisfies 
\[    \lambda^{1/2} \norm{u}{\Leb{p}(\Omega)}+\norm{u}{\Leb{p}(\Omega)}\leq 2C_0\left(\lambda^{-1/2}\norm{f}{\Leb{p}(\Omega)}+\norm{\boldF}{\Leb{p}(\Omega)} \right). \]
This completes the proof of (i). Following the exactly same argument, one can also prove (ii) whose proof is omitted. This completes the proof of Theorem \ref{thm:large-lambda}. 
 \end{proof}

\section{Uniqueness of weak solutions for the problem \eqref{eq:nondivergence-type}}\label{sec:uniqueness}
This section is devoted to a proof of the uniqueness part of Theorem \ref{thm:main-theorem} (i). Below is the main theorem of this section.
\begin{thm}\label{thm:uniqueness}
Let $\Omega$ be a bounded domain in $\mathbb{R}^2$, $2<p<\infty$. Suppose that $A$ satisfies \eqref{eq:uniformly-elliptic}, $\Div A ,\boldb \in \Leb{2}(\Omega;\mathbb{R}^2)$, and $\lambda \geq 0$.  Then there exists a constant $\gamma=\gamma(p,\delta, \norm{\Div A}{\Leb{2}(\Omega)},\norm{\boldb}{\Leb{2}(\Omega)})$ such that under Assumption \ref{assump:leading-coefficient} $(\gamma)$, if $u\in \oSob{1}{p}(\Omega)$ satisfies 
\begin{equation}\label{eq:zero-solution} 
\int_\Omega A\nabla u \cdot \nabla \phi  +(\boldb\cdot \nabla u +\lambda u)\phi  \myd{x}=0 
\end{equation}
for all $\phi \in C_c^\infty(\Omega)$, then $u$ is identically zero in $\Omega$. 
\end{thm}

To prove Theorem \ref{thm:uniqueness}, we use recent results due to Krylov \cite{krylov2020stochastic,Krylov21}.  To state these results, for $0<\delta<1$, let $\mathbb{S}_\delta$ be the set of $n\times n$ real symmetric matrices which are measurable and whose eigenvalues are in $[\delta,\delta^{-1}]$ and $\boldb :\Omega\rightarrow \mathbb{R}^n$ be a vector field.  Write 
\[   Lu =\sum_{i,j=1}^n a^{ij}D_{ij} u -\boldb \cdot \nabla u. \]
The following theorem can be found in \cite[Corollary 3.1]{krylov2020stochastic}, which generalizes the classical theorem due to Alexsandrov \cite{A63}.

\begin{thm}\label{thm:Alexsandrov}
Let $\Omega$ be a bounded domain in $\mathbb{R}^n$ and $c$ be a nonnegative measurable function on $\Omega$. Suppose that   $0<\delta<1$ and  $A$ is  a $\mathbb{S}_{\delta}$-valued function on $\mathbb{R}^n$ and $\boldb \in \Leb{n}(\Omega;\mathbb{R}^n)$. Then there exists a number $n/2<n_0<n$ depending on $n$, $\delta$, and $\norm{\boldb}{\Leb{n}(\Omega)}$ such that if $n_0\leq p<\infty$, then there exists a constant $C$ depending on $n$, $p$, $\delta$, $\norm{\boldb}{\Leb{n}(\Omega)}$, and the diameter of $\Omega$ such that  
\begin{equation}\label{eq:estimate-Alexsandrov} 
u(x)\leq C\norm{(Lu-cu)_{-}}{\Leb{p}(\Omega)}+\sup_{\partial \Omega} u_+ \quad \text{in } \Omega 
\end{equation}
for all $u\in \Sobloc{2}{p}(\Omega)\cap C(\overline{\Omega})$. 
\end{thm}

The following theorem is a special case of \cite[Theorem 4.2]{Krylov21}. 
\begin{thm}\label{thm:Krylov21}
Let $\Omega$ be a bounded $C^{1,1}$-domain in $\mathbb{R}^n$, $n\geq 2$, $1<p<n$, and $\lambda \geq 0$.  Assume that $A$ satisfies \eqref{eq:uniformly-elliptic} and $\boldb \in \Leb{n}(\Omega;\mathbb{R}^n)$.  Then there exists $\gamma=\gamma(n,p,\delta)>0$ such that under Assumption \ref{assump:leading-coefficient} $(\gamma)$, for every $g\in \Leb{p}(\Omega)$, there exists a unique strong solution $u\in \oSob{1}{p}(\Omega)\cap \Sob{2}{p}(\Omega)$ satisfying 
\[   -\sum_{i,j=1}^n a^{ij}D_{ij} u +\boldb \cdot \nabla u+\lambda u =g\quad \text{in } \Omega.\]
\end{thm}
 

Now we are ready to prove the main theorem of this section. 
 
\begin{proof}[Proof of Theorem \ref{thm:uniqueness}]
By Theorem \ref{thm:Alexsandrov}, there exists a number $1<n_0<2$ depending on $\delta$, $\norm{\Div A}{\Leb{2}(\Omega)}$, and $\norm{\boldb}{\Leb{2}(\Omega)}$ such that for $n_0<q<2$, there exists a constant $C$ depending on $q$, $\delta$, $\norm{\Div A}{\Leb{2}(\Omega)}$,  $\norm{\boldb}{\Leb{2}(\Omega)}$, and the diameter of $\Omega$ such that 
\[  u(x)\leq C \norm{(Lu-\lambda u)_{-}}{\Leb{q}(\Omega)}+\sup_{\partial \Omega} u_+\quad \text{in } \Omega  \]
for all $u\in \Sobloc{2}{q}(\Omega)\cap C(\overline{\Omega})$, where
\[  Lu=\sum_{i,j=1}^2 a^{ij}D_{ij}u-(\boldb-\Div A)\cdot \nabla u.\]
Define $q=\frac{1}{2}\left(\max\{2p/(p+2),n_0\}+2\right)$. Then $\max\{n_0,\frac{2p}{p+2}\}<q<2$.  Since $u\in \oSob{1}{p}(\Omega)$ and $p>2$, it follows from  the Sobolev embedding theorem that $u\in C(\overline{\Omega})$. Hence  it suffices to show that $u\in \Sobloc{2}{q}(\Omega)$. To show this, we fix a bounded smooth subdomain  $\Omega'\Subset \Omega$ and let $\zeta \in C_c^\infty(\Omega')$. Since $u$ satisfies \eqref{eq:zero-solution}, we have 
\begin{equation}\label{eq:weak-sol-cut-off}
\int_\Omega A\nabla u \cdot \nabla (\zeta \phi)\myd{x}+\int_\Omega (\boldb \cdot \nabla u+\lambda u)(\zeta \phi)\myd{x}=0
\end{equation}
for all $\phi \in C_c^\infty(\Omega)$.   Since $\Div A \in \Leb{2}(\Omega;\mathbb{R}^2)$ and $u\in \oSob{1}{p}(\Omega)$, $p>2$, it follows that 
 \begin{align}\label{eq:divA} 
 -\int_\Omega \Div A \cdot (u\phi \nabla \zeta) \myd{x}&=\sum_{i,j=1}^2 \int_\Omega a^{ij} D_i (u\phi D_j \zeta) dx \\\nonumber
 &=\sum_{i,j=1}^2 \int_\Omega a^{ij} (D_i u)(D_j \zeta) \phi \myd{x}+\sum_{i,j=1}^2 \int_\Omega u a^{ij}(D_i \phi)(D_j \zeta) dx \\\nonumber
 &\relphantom{=}+\sum_{i,j=1}^2 \int_\Omega u\phi a^{ij}D_{ij} \zeta \myd{x} +\int_\Omega u\phi\left( \sum_{i,j=1}^2 a^{ij} D_{ij} \zeta  \right) \myd{x}  \nonumber
 \end{align} 
 for all $\phi \in C_c^\infty(\Omega)$. Hence by \eqref{eq:weak-sol-cut-off} and \eqref{eq:divA}, we get  
 \begin{equation}\label{eq:cut-off}
 \begin{aligned}
 &\int_\Omega A\nabla (u\zeta) \cdot \nabla \phi  +  (\boldb \cdot \nabla (\zeta u)+\lambda_1 (\zeta u))\phi \myd{x}=-\int_\Omega g \phi \myd{x},
 \end{aligned}
 \end{equation} 
where 
\[   g= u \Div A \cdot \nabla \zeta+\sum_{i,j=1}^2(a^{ij} D_{ij} \zeta ) u + \nabla u \cdot  A\nabla \zeta +A\nabla u \cdot \nabla \zeta  -u\boldb \cdot \nabla \zeta   + (\lambda-\lambda_1)(\zeta u).\]
Note that $g\in \Leb{2}(\Omega)$ since $u\in \oSob{1}{p}(\Omega)$, $p>2$, and $\Div A,\boldb \in \Leb{2}(\Omega;\mathbb{R}^2)$. 

By Theorem \ref{thm:large-lambda}, there exist   $\gamma_0=\gamma_0(p,\delta)>0$ and   $\lambda_1=\lambda_1(p,\delta,\boldb,\Omega')\geq 1$ such that under Assumption \ref{assump:leading-coefficient} $(\gamma_0)$, for $\lambda \geq \lambda_1$, $f\in \Leb{p}(\Omega')$, and $\boldF \in \Leb{p}(\Omega';\mathbb{R}^2)$, there exists a unique weak solution $v\in \oSob{1}{p'}(\Omega)$ such that 
\[    -\Div (A \nabla v)+\boldb \cdot \nabla v +\lambda v=f+\Div \boldF\quad \text{in } \Omega',\quad v=0\quad \text{on } \partial\Omega'. \]

For such $\lambda_1$ and $q=\frac{1}{2}\left(\max\{2p/(p+2),n_0\}+2\right)$, it follows from Theorem \ref{thm:Krylov21} that there exists $\gamma_1=\gamma_1(p,n_0,\delta)>0$ such that under Assumption \ref{assump:leading-coefficient} $(\gamma_1)$, there exists a unique $w\in \oSob{1}{q}(\Omega')\cap \Sob{2}{q}(\Omega')$ satisfying 
\[   -\sum_{i,j=1}^2 a^{ij}D_{ij}w+(\boldb-\Div A)\cdot \nabla w+\lambda_1 w=-g\quad \text{in } \Omega'.  \]
Set $\gamma=\min\{\gamma_0,\gamma_1\}$ and choose $A$ so that $A$  satisfies \eqref{eq:uniformly-elliptic}, $\Div A \in \Leb{2}(\Omega;\mathbb{R}^2)$, and Assumption \ref{assump:leading-coefficient} $(\gamma)$. Since $w\in \oSob{1}{q}(\Omega')\cap\Sob{2}{q}(\Omega')$ and $\Div A \in \Leb{2}(\Omega;\mathbb{R}^2)$, it follows from the Sobolev embedding theorem that  $w\in \oSob{1}{q^*}(\Omega')$  and
\begin{align*} 
 \int_{\Omega'} \Div A \cdot (\phi \nabla w) \myd{x}&=-\sum_{i,j=1}^2\int_{\Omega'} a^{ij}D_i (\phi D_j w)\myd{x}\\
 &=-\sum_{i,j=1}^2\int_{\Omega'} (a^{ij}D_{ij} w)\phi \myd{x}-\int_{\Omega'} A\nabla w \cdot \nabla \phi\myd{x}
 \end{align*}
for all $\phi \in C_c^\infty(\Omega')$. Hence $w$ satisfies  
\[    \int_{\Omega'} A \nabla w \cdot \nabla \phi +(\boldb \cdot \nabla w +\lambda_1 w)\phi \myd{x}=-\int_{\Omega'} g  \phi \myd{x}   \]
for all $\phi \in C_c^\infty(\Omega')$.   Since $2p/(p+2)<q<2$ and $\Omega'$ is bounded, it follows that $\oSob{1}{q^*}(\Omega')\subset \oSob{1}{p}(\Omega')$. Hence by \eqref{eq:cut-off} and Theorem \ref{thm:large-lambda}, we have  $w=\zeta u$ in $\Omega'$, which implies that $\zeta u \in \Sob{2}{q}(\Omega')$ for any $\zeta \in C_c^\infty(\Omega')$.  Since there exists a sequence of smooth bounded subdomains $\{\Omega_k\}$ satisfying $\Omega_k\Subset \Omega_{k+1}\Subset \Omega$  and $\bigcup_{k} \Omega_k = \Omega$ (see e.g. \cite[Proposition 8.2.1]{D08}),  we conclude that $u\in \Sobloc{2}{q}(\Omega)$.

 Since $\Div A \in \Leb{2}(\Omega;\mathbb{R}^2)$, we have 
\[    \sum_{i,j=1}^2\int_{\Omega} a^{ij} D_i \left[ (D_j u)\phi \right]\myd{x} = -\int_\Omega \phi \Div A \cdot \nabla u\myd{x}  \]
for all $\phi \in C_c^\infty(\Omega)$. Hence  $u$ satisfies 
\[  -\sum_{i,j=1}^2a^{ij}D_{ij} u + (\boldb  - \Div A)\cdot \nabla u +\lambda u =0\quad \text{in } \Omega.  \]Therefore  by Theorem \ref{thm:Alexsandrov}, we conclude that $u$ is identically zero in $\Omega$. This completes the proof of Theorem \ref{thm:uniqueness}. 
\end{proof} 

\section{Proof of Theorem \ref{thm:main-theorem}}\label{sec:completion}
This section is devoted to a proof of Theorem \ref{thm:main-theorem}. For the case of sufficiently large $\lambda$, this was done by Theorem \ref{thm:large-lambda}. To treat the case of small $\lambda$, we use a functional analytic argument as in Kang-Kim \cite{KK15,KK17} and Kim-Kwon \cite{KK18} to deduce that it suffices to prove the uniqueness part of (i). To explain this, let  $2<p<\infty$ be fixed. Then 
 \[   (u,v)\mapsto\action{\mathcal{L}_p u,v} = \int_\Omega (A\nabla u) \cdot \nabla v \myd{x}+\int_\Omega (\boldb\cdot \nabla u)v\myd{x} \]
 and 
 \[  (u,v)\mapsto \action{\mathcal{L}_{p'}^* v,u} = \int_\Omega (A^T\nabla v) \cdot \nabla u\myd{x}+\int_\Omega (v\boldb)\cdot\nabla u  \myd{x}=\action{\mathcal{L}_p u,v} \]
 define  bounded linear operators
 \[   \mathcal{L}_p : \oSob{1}{p}(\Omega)\rightarrow \Sob{-1}{p}(\Omega)\quad \text{and}\quad \mathcal{L}_{p'}^* : \oSob{1}{p'}(\Omega)\rightarrow \Sob{-1}{p'}(\Omega).  \]
 Let $\mathcal{I}_p : \Leb{p}(\Omega)\rightarrow (\Leb{p'}(\Omega))'$ be the isomorphism defined by 
 \[   \action{\mathcal{I}_p u,v} = \int_\Omega u v \myd{x}\quad \text{for all }  (u,v) \in \Leb{p}(\Omega)\times \Leb{p'}(\Omega). \]
We also define $\Div_p : \Leb{p}(\Omega;\mathbb{R}^2) \rightarrow \Sob{-1}{p}(\Omega)$ by 
\[   \action{\Div_p \boldF,\varphi} = -\int_\Omega \boldF \cdot \nabla \varphi \myd{x}\quad \text{for all } \varphi \in \oSob{1}{p'}(\Omega).\]
Then it is easy to show that $u\in \oSob{1}{p}(\Omega)$ is a weak solution of the problem \eqref{eq:nondivergence-type} if and only if $(\mathcal{L}_p + \lambda \mathcal{I}_p)u=\mathcal{I}_p f+\Div_p \boldF$, while $v\in \oSob{1}{p'}(\Omega)$ is a weak solution of \eqref{eq:divergence-type} if and only if $(\mathcal{L}_{p'}^*+\lambda \mathcal{I}_{p'})v=\mathcal{I}_{p'}g + \Div_{p'}\mathbf{G}$. 

By Theorem \ref{thm:large-lambda}, there exist $\gamma=\gamma(p,\delta)>0$, $\theta=\theta(p,\delta)>0$, and  $\lambda_1\geq 1$ such that under Assumptions \ref{assump:leading-coefficient} $(\gamma)$ and \ref{assump:domain-regularity} $(\theta)$, the operators $\mathcal{L}_p+\lambda_1 \mathcal{I}_p$ and $\mathcal{L}_{p'}^*+\lambda_1 \mathcal{I}_{p'}$ are invertible. Then for $0\leq \lambda \neq \lambda_1$, the linear operators 
\[   \mathcal{K}_{p,\lambda} = (\lambda_1-\lambda)(\mathcal{L}_p+\lambda_1 \mathcal{I}_{p})^{-1}\circ \mathcal{I}_p : \Leb{p}(\Omega)\rightarrow \Leb{p}(\Omega) \]
and 
\[   \mathcal{K}_{p',\lambda}^* = (\lambda_1-\lambda)(\mathcal{L}_{p'}^*+\lambda_1 \mathcal{I}_{p'})^{-1}\circ \mathcal{I}_{p'} : \Leb{p'}(\Omega)\rightarrow \Leb{p'}(\Omega) \]
are bounded and even compact by Rellich-Kondrachov's embedding theorem. Since 
\[   \int_\Omega (\mathcal{K}_{p,\lambda} u)v \myd{x} = \int_\Omega (\mathcal{K}_{p',\lambda}^* v)u \myd{x} \]
for all $(u,v) \in \Leb{p}(\Omega)\times \Leb{p'}(\Omega)$, it follows that 
\begin{equation}\label{eq:duality-relation-K}
\mathcal{K}_{p',\lambda}^* = \mathcal{I}_{p'}^{-1}\circ \mathcal{K}_{p,\lambda}'\circ \mathcal{I}_{p'},  
\end{equation}
where $\mathcal{K}_{p,\lambda}' : (\Leb{p}(\Omega))'\rightarrow (\Leb{p}(\Omega))'$ is the adjoint operator of $\mathcal{K}_{p,\lambda}$. Note also that 
\[   \ker (Id-\mathcal{K}_{p,\lambda})=\ker (\mathcal{L}_p+\lambda \mathcal{I}_p)\quad \text{and}\quad \ker (Id-\mathcal{K}_{p',\lambda}^*)=\ker (\mathcal{L}_{p'}^*+\lambda \mathcal{I}_{p'}). \]
By \eqref{eq:duality-relation-K} and the Fredholm alternative theorem (see e.g. \cite[Theorem 6.6]{Brezis11}), we have 
\begin{equation}\label{eq:image-duality-1}
 \mathrm{Im}(Id-\mathcal{K}_{p,\lambda}) = \left\{ u\in \Leb{p}(\Omega) : \int_\Omega uv \myd{x}=0\quad \text{for all } v \in \ker (\mathcal{L}_{p'}^*+\lambda \mathcal{I}_{p'}) \right\},  
\end{equation}
\begin{equation}\label{eq:image-duality-2} \mathrm{Im}(Id-\mathcal{K}_{p',\lambda}^*) = \left\{ v\in \Leb{p'}(\Omega) : \int_\Omega vu \myd{x}=0\quad \text{for all } u \in \ker (\mathcal{L}_p+\lambda \mathcal{I}_p) \right\},  
\end{equation}
and 
\begin{equation}\label{eq:kernel-characterization} 
  \dim\ker (\mathcal{L}_{p}+\lambda \mathcal{I}_{p})=\dim  \ker (\mathcal{L}_{p'}^*+\lambda \mathcal{I}_{p'}) <\infty 
\end{equation}
for all $\lambda \geq 0$. 

Now we are ready to prove Theorem \ref{thm:main-theorem}. 
\begin{proof}[Proof of Theorem \ref{thm:main-theorem}]
Choose $(\gamma_0,\theta_0,\lambda_1)$ from Theorem \ref{thm:large-lambda} and choose $\gamma_1$ from Theorem \ref{thm:uniqueness}, and let $\gamma=\min\{\gamma_0,\gamma_1\}$.  Under Assumptions \ref{assump:leading-coefficient} $(\gamma)$ and \ref{assump:domain-regularity} $(\theta)$,  Theorem \ref{thm:main-theorem} holds for all  $\lambda \geq \lambda_1$ by Theorem \ref{thm:large-lambda}.     Suppose that $0\leq \lambda <\lambda_1$. By  \eqref{eq:kernel-characterization} and Theorem \ref{thm:uniqueness}, we have 
\begin{equation}\label{eq:uniqueness-kernel}
  \ker (\mathcal{L}_{p}+\lambda \mathcal{I}_p)=\ker(\mathcal{L}_{p'}^*+\lambda  \mathcal{I}_{p'})=\{0\} 
\end{equation}
for any $\lambda \geq 0$.  Hence by \eqref{eq:image-duality-1} and \eqref{eq:image-duality-2}, we have 
\begin{equation}\label{eq:image-Lp} 
  \mathrm{Im}(Id-\mathcal{K}_{p,\lambda})=\mathrm{Im}(Id-\mathcal{K}_{p',\lambda}^*)=\Leb{p}(\Omega). 
\end{equation}
Given $f\in \Leb{p}(\Omega)$ and $\boldF \in \Leb{p}(\Omega;\mathbb{R}^2)$, let 
\[    w=(\mathcal{L}_{p}+\lambda_1  \mathcal{I}_{p})^{-1} (\mathcal{I}_pf+\Div_p \boldF).\]
By \eqref{eq:image-Lp},  there exists $u\in \Leb{p}(\Omega)$ such that 
\[     u-\mathcal{K}_{p,\lambda}u=w.\]
Since $(\mathcal{L}_{p}+\lambda_1  \mathcal{I}_{p})^{-1}$ maps $\Sob{-1}{p}(\Omega)$ to $\oSob{1}{p}(\Omega)$, it follows from the definition of $\mathcal{K}_{p,\lambda}$ that $w, \mathcal{K}_{p,\lambda} u \in \oSob{1}{p}(\Omega)$. Hence 
\[    u\in \oSob{1}{p}(\Omega)\quad \text{and}\quad (\mathcal{L}_{p}+\lambda \mathcal{I}_{p})u=\mathcal{I}_pf+\Div_p \boldF. \]  
This proves that $u$ is a weak solution of the problem \eqref{eq:nondivergence-type}. Similarly, we can also prove the existence of weak solution of the problem \eqref{eq:divergence-type}. By \eqref{eq:uniqueness-kernel}, we also have uniqueness of weak solution of the problem \eqref{eq:divergence-type}. It remains to show the desired estimate. 

Let $f\in \Leb{p}(\Omega)$ and $\boldF \in \Leb{p}(\Omega;\mathbb{R}^2)$. Define 
\[  \action{\ell,\phi}=\int_\Omega f \phi \myd{x}-\int_\Omega \boldF \cdot \nabla \phi \myd{x}\quad \text{for all }\phi \in C_c^\infty(\Omega).\]
Then 
\[  \ell \in \Sob{-1}{p}(\Omega)\quad \text{and}\quad \norm{\ell}{\Sob{-1}{p}(\Omega)}\leq \norm{f}{\Leb{p}(\Omega)}+\norm{\boldF}{\Leb{p}(\Omega)}. \]
Since the operator $\mathcal{L}_p+\lambda \mathcal{I}_p$ is bijective, it follows that there exists $u_\lambda\in \oSob{1}{p}(\Omega)$ such that 
\[    (\mathcal{L}_p +\lambda \mathcal{I}_p) u_\lambda =\ell. \]
Moreover, it follows from the bounded inverse theorem  (see e.g. \cite[Corollary 2.7]{Brezis11}) that there exists a constant $C_{\lambda}>0$ independent of $u_\lambda$ and $\ell$ such that 
\[    \norm{u_\lambda}{\Sob{1}{p}(\Omega)}\leq C_\lambda \norm{\ell}{\Sob{-1}{p}(\Omega)}\leq C_\lambda (\norm{f}{\Leb{p}(\Omega)}+\norm{\boldF}{\Leb{p}(\Omega)}). \]

In fact, we can choose a constant independent of $\lambda$. Indeed, for $0\leq \lambda \leq  \lambda_1$, let 
\[ E_\lambda = \left\{ \mu \in [0,\lambda_1] : C_\lambda |\lambda -\mu|<\frac{1}{2}\right\}. \]
For $\mu \in E_{\lambda}$, since the operator $\mathcal{L}_{p}+\lambda \mathcal{I}_p$ is bijective,  there exists $u_\mu \in \oSob{1}{p}(\Omega)$ such that 
\[     (\mathcal{L}_{p}+\lambda \mathcal{I}_p) u_\mu =\mathcal{I}_pf + \Div_p \boldF  +(\lambda -\mu)\mathcal{I}_pu_{\mu}. \]
Moreover,  we have 
\[   \norm{u_{\mu}}{\Sob{1}{p}(\Omega)}\leq 2C_\lambda (\norm{f}{\Leb{p}(\Omega)}+\norm{\boldF}{\Leb{p}(\Omega)})  \]
for all $\mu \in E_\lambda$. Hence by the compactness of $[0,\lambda_1]$, there is a constant $C>0$ independent of $u_\lambda$, $f$, $\boldF$, and $\lambda$ such that 
\[   \norm{u_\lambda}{\Sob{1}{p}(\Omega)}\leq 2C(\norm{f}{\Leb{p}(\Omega)}+\norm{\boldF}{\Leb{p}(\Omega)}).   \]
By the Poincar\'e inequality, we conclude that 
\[   \norm{\nabla u_\lambda}{\Leb{p}(\Omega)}+\lambda^{1/2}\norm{u_\lambda}{\Leb{p}(\Omega)}\leq C (\norm{f}{\Leb{p}(\Omega)}+\norm{\boldF}{\Leb{p}(\Omega)}) \]
for all $\lambda \in [0,\lambda_1]$. 
This proves the desired estimate for a solution of the problem \eqref{eq:nondivergence-type}. Following exactly the same argument, one can derive a similar estimate for a weak solution of the problem \eqref{eq:divergence-type}. This completes the proof of Theorem \ref{thm:main-theorem}.
\end{proof} 
\bibliographystyle{amsplain} 
\providecommand{\bysame}{\leavevmode\hbox to3em{\hrulefill}\thinspace}
\providecommand{\MR}{\relax\ifhmode\unskip\space\fi MR }
\providecommand{\MRhref}[2]{%
  \href{http://www.ams.org/mathscinet-getitem?mr=#1}{#2}
}
\providecommand{\href}[2]{#2}

\end{document}